%% file: arvixMIFpaper.tex
\documentclass{proc-l}

\usepackage{amsthm}
\usepackage{amsmath}
\usepackage{amssymb}
\usepackage{amsfonts}
\usepackage{amsxtra} 
\usepackage{mathrsfs}
\usepackage{color}
\usepackage{graphicx}
\usepackage[all]{xy}

\newtheorem{theorem}{Theorem}[section]
\newtheorem{lemma}[theorem]{Lemma}
\newtheorem{corollary}[theorem]{Corollary}
\theoremstyle{definition}
\newtheorem{definition}[theorem]{Definition}

\newtheorem{question}[theorem]{Question}
\theoremstyle{remark}
\newtheorem{remark}[theorem]{Remark}
\newtheorem{notation}[theorem]{Notation}
\numberwithin{equation}{section}

\include{newcommands}

\newcommand{\ii}{{\mathfrak{i}}}
\newcommand{\iicl}{{\mathfrak{i}_{\rm cl}}}
\newcommand{\iiB}{{\mathfrak{i}_{B}}}
\newcommand{\aaa}{{\mathfrak{a}}}
\newcommand{\aacl}{{\mathfrak{a}_{\rm cl}}}
\newcommand{\aaB}{{\mathfrak{a}_{B}}}
\newcommand{\bbb}{{\mathfrak{b}}}

\renewcommand{\aa}{{a_1, \dots, a_n}}
\renewcommand{\bb}{{b_1, \dots, b_\ell}}

\newcommand{\bbP}{{\mathbb{P}}}

\newcommand{\calA}{{\mathscr{A}}}
\renewcommand{\A}{{\calA}}
\renewcommand{\B}{{\mathscr{B}}}
\newcommand{\al}{{\alpha}}
\newcommand{\calE}{{{\mathcal{E}}}}
\newcommand{\FF}{{\rm FF}}

\begin{document}

\title{Definable Maximal Independent Families}

\author[Brendle]{J\"org Brendle}
\address{Graduate School of System Informatics, Kobe University, Rokkodai 1-1, Nada, Kobe 657-8501, Japan} 
\email{brendle@kobe-u.ac.jp}
 
\author[Fischer]{Vera Fischer}
 \address{Kurt G\"odel Research Center, Universit\"at Wien, W\"ahringer Stra{\ss}e 25, 1090 Vienna, Austria}
\email{vera.fischer@univie.ac.at}

\author[Khomskii]{Yurii Khomskii}
\address{Universit\"at Hamburg, Fachbereich Mathematik, Bundesstra{\ss}e 55, 20146 Hamburg, Germany}
\email{yurii@deds.nl}

\thanks{Partially supported by Grants-in-Aid for Scientific Research (C) 15K04977 and 18K03398, Japan Society for the Promotion of Science (Brendle); the Austrian Science Foundation (FWF) by the START Grant number Y1012-N35 (Fischer); and  the European Commission under a Marie Curie Individual Fellowship (H2020-MSCA-IF-2015) through the project number 706219, acronym REGPROP (Khomskii).   Partially supported by the Isaac Newton Institute for Mathematical Sciences in the programme Mathematical, Foundational and Computational Aspects of the Higher Infinite (HIF) funded by EPSRC grant EP/K032208/1 (Brendle and Khomskii).}

\subjclass[2010]{03E15, 03E17, 03E35}



\begin{abstract} We study maximal independent families (m.i.f.) in the projective hierarchy. We show that (a) the existence of a $\SIGMA^1_2$ m.i.f.\ is equivalent to the existence of a $\PI^1_1$ m.i.f., (b) in the Cohen model, there are no projective maximal independent families, and (c) in the Sacks model, there is a $\PI^1_1$ m.i.f. We also  consider a new cardinal invariant related to the question of destroying or preserving maximal independent families. \end{abstract}

\maketitle

\section{Introduction}

In descriptive set theory, a recurrent line of inquiry is whether objects  defined in a non-constructive way can exist on given levels of the projective hierarchy. For example, an ultrafilter cannot be $\SIGMA^1_1$. Mathias \cite{HappyFamilies} proved that there are no $\SIGMA^1_1$ maximal almost disjoint (mad) families. If $V=L$ then there are $\DELTA^1_2$ ultrafilters and $\PI^1_1$ mad families. 

In this paper, we look at \emph{maximal independent families} from the definable point of view.  Our main results are Theorem \ref{equivalence}, Theorem \ref{d} and Theorem \ref{mainsackstheorem}, stating that the existence of a $\SIGMA^1_2$ m.i.f. is equivalent to the existence of a $\PI^1_1$ m.i.f., that in the Cohen model there are no projective m.i.f.'s, and that in the Sacks model there is a $\PI^1_1$ m.i.f., respectively.

\begin{definition} A family $\I \subseteq \wuw$ is  \emph{independent} if for all $a_1, \dots, a_n \in \I$ and different $b_1, \dots, b_\ell \in \I$
$$a_1 \cap \dots a_n \cap (\omega \setminus b_1) \cap \dots \cap (\omega \setminus b_\ell) \text{ is infinite.}$$   A family $\I \subseteq \wuw$ is called a \emph{maximal independent family} (\emph{m.i.f.}) if it is independent and maximal with regard to this property. \end{definition}

We will typically use the following abbreviation:
$$\sigma(\bar{a};\bar{b}) := a_1 \cap \dots a_n \cap (\omega \setminus b_1) \cap \dots \cap (\omega \setminus b_\ell)$$
where it will be assumed that all of the $a_i$ are different from all of the $b_j$. Note that maximality of $\I$ is equivalent to:
$$\forall X \in \wuw \; \exists \aa, \bb \in \I \text{ s.t. }\;  \sigma(\bar{a}; \bar{b}) \subseteq^* X \; \text{ or } \;   \sigma(\bar{a}; \bar{b})  \cap X  =^* \varnothing.$$

By identifying the space $\wuw$ with $\dw$ via characteristic functions, one can consider independent families as subsets of the reals and study their complexity in the projective hieararchy.

\begin{lemma} If $\I$ is a $\SIGMA^1_n$ m.i.f. then it is $\DELTA^1_n$. \end{lemma}

\begin{proof} Suppose $\I$ is a $\SIGMA^1_n$ m.i.f. Then  $X \notin \I$ iff: 
$$ \exists \aa, \bb \in \I \; \text{ s.t. } X \notin \{\aa, \bb\}  \; \text{ and }$$
$$  \sigma(\bar{a},\bar{b}) \subseteq^* X \; \text{ or } \;  \sigma(\bar{a},\bar{b})  \cap X =^* \varnothing. $$  This  statement is easily seen to be $\SIGMA^1_n$.  \end{proof}

\begin{theorem}[Miller; {\cite{MillerPI11}}] There is no analytic m.i.f. \end{theorem}

An analysis of Miller's proof shows that it really only uses the Baire property of analytic sets. In particular, if we use $\SIGMA^1_n(\IC)$ to denote the statement ``all $\SIGMA^1_n$ sets have the Baire property'', then Miller's proof shows that for any $n$, $\SIGMA^1_n(\IC)$ implies that there are no $\SIGMA^1_n$ m.i.f. (see Corollary~\ref{corb}). It follows that 
\begin{enumerate}
\item In the Hechler model, as well as the Amoeba or Amoeba-for-category model, there is no $\SIGMA^1_2$ m.i.f.,
\item In the Solovay model (the L\'evy-collapse of an inaccessible), as well as Shelah's model for the Baire Property without inaccessibles \cite{ShelahInaccessible}, there is no projective m.i.f.,
\item In $L(\IR)$ of the above two models, there is no m.i.f. at all, and
\item   $\AD \Rightarrow $ there is no  m.i.f. (this follows because under $\AD$ all sets of reals have the property of Baire, see again~\ref{corb}).
\end{enumerate}

In this paper, we   prove a stronger result, namely, that in the Cohen model there is no projective m.i.f., and in the $L(\IR)$ of the Cohen model there is no m.i.f. at all. Notice that since  $\SIGMA^1_2(\IC)$ is false in the Cohen model, this shows that the converse implication ``$\SIGMA^1_n(\IC)\; \Leftarrow \;\nexists \SIGMA^1_n$-m.i.f.'' consistently fails.

\bigskip

On the other hand, it is easy to construct a m.i.f. by induction using a wellorder of the reals, and thus, it is not hard to see that  in $L$ there exists a $\SIGMA^1_2$ m.i.f. In \cite{MillerPI11} Miller used sophisticated coding techniques to show that, in fact, this m.i.f. can be constructed in a  $\PI^1_1$ fashion. Building on an idea of Asger T\"ornquist \cite{Tornquist}, we  show that in fact this proof is unnecessary since one can derive, directly in $\ZFC$, that if there exists a $\SIGMA^1_2$ m.i.f. then there exists a $\PI^1_1$ m.i.f.

\bigskip A construction originally attributed to Eisworth and Shelah (see \cite{HandbookCardinals}), implicitly appearing in Shelah's proof of $\mathfrak{i} < \mathfrak{u}$ \cite{ShelahCon} and elaborated  in \cite{FischerMontoya}, yields a  forcing for generically adding a Sacks-indestructible m.i.f.   In this paper we show that this family can be defined in a $\SIGMA^1_2$ way in $L$. Therefore, in the countable support iteration of Sacks forcing, as well as the product of Sacks forcing, starting from $L$, there exists a $\SIGMA^1_2$ m.i.f., and hence a $\PI^1_1$ m.i.f. In fact, a slight modification produces a family which is  indestructible by the poset used in \cite{ShelahCon}, which shows that the  consistency of $\mathfrak{i}<\mathfrak{u}$ can be witnessed by a $\PI^1_1$ m.i.f.



\section{$\SIGMA^1_2$ and $\PI^1_1$ m.i.f's} \label{SectionPI11}

\begin{theorem} \label{equivalence} If there exists a $\SIGMA^1_2$ m.i.f. then there exists a $\PI^1_1$ m.i.f. \end{theorem}

\begin{proof} Suppose $\I_0$ is a $\SIGMA^1_2$ maximal independent family. Let $F_0 \subseteq \left(\wuw\right)^2$ be a $\PI^1_1$ set such that $\I_0$ is the projection of $F_0$. Consider the space $\omega \; \dot{\cup}\;  \dlw$ as a disjoint union, and consider the mapping
	$$  g: \begin{array}{l}    \left(\wuw\right)^2  \longrightarrow \PP\left(\omega   \;  \dot{\cup} \;  \dlw \right) \\   (x,y) \longmapsto x \cup \{\chi_y \till n \; \mid \; n < \omega \}
	\end{array} $$ where $\chi_y$ is the characteristic function of $y$. It is not hard to see that $g$ is a continuous function (in the sense of the space $\PP(\omega \; \dot{\cup}\;  \dlw))$. 
	
	\p By $\PI^1_1$-uniformization, there exists a $\PI^1_1$ set $F \subseteq F_0$ which is the graph of a function, i.e., $\forall x \in \I_0 \; \exists ! y \; ((x,y) \in F)$. We let $\I := g[F]$ and claim that $\I$ is a $\PI^1_1$ m.i.f.
	
	\p To see that $\I$ is $\PI^1_1$, note that for $z \in  [ \omega   \;  \dot{\cup} \;  \dlw ]^\omega$, there is an explicit way to recover $x$ and $y$ such that $g(x,y) = z$, if such $x$ and $y$ exist. More precisely:   $z \in \I$ if and only if 
\begin{enumerate}
\item $z \cap \dlw$ is a single branch, \\ i.e., $\forall n \; \exists ! s \in z \cap 2^n$ and $\forall s, t \in z \cap \dlw \; (|s| < |t| \;\to\; s \subset t )$, 
\item $\forall y \; (\forall n \; (y \till n \in z \cap \dlw) \; \to \; (z \cap \omega, y) \in F)$. \end{enumerate}
This gives a $\PI^1_1$ definition of $\mathcal{I}$.\footnote{Following the suggestion of the anonymous referee, an alternative argument is as follows: Extend $g$ to $g'$ defined on the compact set $\PP(\omega)^2$ in such a way that $g'$ is injective. Then $g'$ is a  homeomorphism between $\PP(\omega)^2$ and the compact set $g'[\PP(\omega)^2]$, thus $\I = g'[F] = g[F] $ is a $\PI^1_1$ subset of the compact set $g'[\PP(\omega)^2]$.}

	\p To see that $\I$ is independent, suppose we have $z_1, \dots , z_n$ and $w_1, \dots , w_\ell \in \I$, the $z$'s being different from the $w$'s. Write $a_i := z_i \cap \omega$ and $b_j := w_j \cap \omega$. Then all $a_i$ and $b_j$ are in $\dom(F) = \I_0$, and moreover, since $F$ is a function, the $a_i$'s are different from the $b_j$'s. But then we have that $\sigma(z_1, \dots, z_n; w_1, \dots, w_\ell) \supseteq \sigma(a_1, \dots, a_n; b_1, \dots, b_\ell)$ is infinite, since the latter set is infinite by the independence of $\I_0$.
	
	\p To show maximality of $\I$, suppose $W \in [ \omega   \;  \dot{\cup} \;  \dlw  ]^\omega$ and $W \notin \I$. Let $A := W \cap \: \omega$. By maximality of $\I_0$, there are $a_1, \dots, a_n \in \I_0$ and different $b_1, \dots, b_\ell \in \I_0$ such that $\sigma(a_1, \dots, a_n, A ; \; b_1, \dots, b_\ell)$ is finite or $\sigma(a_1, \dots, a_n; \; b_1, \dots, b_\ell,A)$ is finite, w.l.o.g. the former. Then there are  $z_1, \dots, z_n$ and different $w_1, \dots, w_\ell$ such that $a_i = z_i \cap \omega$ and $b_j = w_j \cap \omega$. To make sure that the ``$\dlw$-part'' of the $z_i$'s and the $w_j$'s does not make the intersection infinite, we pick two additional $t_0 \neq t_1 \in \I$, different from the $z_i$'s and the $w_j$'s. Let $t_0 = g(x_0, y_0)$ and $t_1 = g(x_1, y_1)$. If $y_0 = y_1$, then $(t_0 \setminus t_1) \cap \dlw = \varnothing$, hence $\sigma(x_1, \dots, x_n, W, t_0 ; \; w_1, \dots, w_\ell, t_1)$ is finite. If, on the other hand, $y_0 \neq y_1$, then the sets $\{\chi_{y_0} \till n \; \mid \; n < \omega\}$ and $\{\chi_{y_1} \till n \; \mid \; n < \omega\}$ are almost disjoint,  so $(t_0 \cap t_1) \cap \dlw$ is finite. In that case, $\sigma(x_1, \dots, x_n, W, t_0, t_1 ; \; w_1, \dots , w_\ell)$ is finite. So in any case, $\I \cap \{W\}$ is not independent, completing the proof.		\end{proof}

Clearly the above proof also holds pointwise for every parameter, i.e., if there is a $\Sigma^1_2(a)$ m.i.f. then there is a $\Pi^1_1(a)$ m.i.f. However, since $\PI^1_1$-uniformisation is essential, the following natural question remains open:

\begin{question} Does the existence of a $\SIGMA^1_n$ m.i.f. imply the existence of a $\PI^1_{n-1}$ m.i.f., for $n > 2$? \end{question}


\section{Projective M.I.F.'s} \label{SectionProjective}

The main result of this section is:

\begin{theorem} \label{d} In the Cohen model $W$ (that is, the model obtained by adding at least $\omega_1$ Cohen reals over any
model of set theory) there are no projective m.i.f.'s. Furthermore, in $L(\IR)^W$, there are no m.i.f.'s. \end{theorem}

The theorem is proved in three steps. First, we isolate a new regularity property which was implicit in Miller's original proof (\cite[proof of 10.28]{MillerPI11}).

\newcommand{\Sadac}{{\mathbb{S}_{ad\text{-}ac}}}
\newcommand{\Sacad}{{\mathbb{S}_{ac\text{-}ad}}}

\begin{definition} A tree $T \subseteq \dlw$ is called \emph{perfect almost disjoint} (\emph{perfect a.d.}) if it is a perfect tree and $\forall x \neq y \in [T]$ the set $\{n \mid x(n) = y(n)=1\}$ is finite. A tree $S \subseteq \dlw$ is called \emph{perfect almost covering}  (\emph{perfect a.c.})  if it is a perfect tree and $\forall x \neq y \in [T]$, the set $ \{n \mid x(n) = y(n)=0\}$ is finite. \end{definition}

\begin{definition} A set $X \subseteq \dw$ satisfies the \emph{perfect-a.d.-a.c.\ property}, abbreviated by $\Sadac$, if there exists a perfect a.d.\ tree $T$ with $[T] \subseteq X$, or there exists a perfect a.c.\ tree $S$ with $[S] \cap X = \varnothing$. \end{definition}

\begin{remark} Note that one could also  define the symmetric property:   $X \subseteq \dw$ satisfies the \emph{perfect-a.c.-a.d.\ property} if there exists a perfect a.c.\ tree $S$ with $[S] \subseteq X$, or there exists a perfect a.d.\ tree $T$ with $[T] \cap X = \varnothing$. A   curious aspect of our proof is that either of these two properties yields the proof in an analogous fashion, but since one of them is sufficient  we pick the former. \end{remark}

\begin{lemma} \label{b} Let $\Gamma$ be a pointclass closed under existential quantification over reals. Then $\Gamma (\Sadac)$ implies that there is no m.i.f. in $\Gamma$.
In particular, $\SIGMA^1_n(\Sadac) \Rightarrow \nexists \SIGMA^1_n\text{-m.i.f.}$ \end{lemma}

\begin{proof}   
Assume $\I$ is a $\Gamma$ m.i.f. Define 
$$ \hskip-0.75cm H := \{X \mid \exists \bar{a}, \bar{b}  \subset \I \;  \text{ disjoint s.t. }  \sigma(\bar{a}; \bar{b}) \subseteq^* X\}$$
$$ K := \{X \mid \exists \bar{a}, \bar{b}  \subset \I   \;  \text{ disjoint s.t. }     \sigma(\bar{a}; \bar{b}) \cap X =^* \varnothing\}. $$ 
The maximality of $\I$ implies that $\wuw = H \cup K$. Moreover, $H$ is a $\Gamma$ set ($K$ also is, but this turns out to be irrelevant). 

\p From  $\Gamma(\Sadac)$ we then obtain a perfect almost disjoint tree $T$ with $[T] \subseteq H$, or a perfect almost covering tree $S$ with $[S] \cap H = \varnothing$, hence $[S] \subseteq K$. Assume the former.

\p For each $X \in [T]$ let $a^X_1, \dots, a^X_{n_X}$ and $b^X_1, \dots, b^X_{\ell_X}$ witness  the fact that $X \in H$. Applying the  $\Delta$-systems lemma to the family $\{\{a^X_1, \dots, a^X_{n_X}, b^X_1, \dots, b^X_{\ell_X}\} \; \mid \; X \in [T]\}$, find an uncountable subset of $[T]$ with a fixed root $R$. Moreover, an uncountable sub-family of this family has the property that the elements of the root $R$ have the same function in the sense of ``being an $a^X_i$'' or ``being a $b^X_j$''. It follows that, for distinct $X,Y$ from this family, we have: 
$$\{a^X_1, \dots, a^X_{n_X}, a^Y_1, \dots, a^Y_{n_Y}\} \; \cap  \; \{b^X_1, \dots, b^X_{\ell_X}, b^Y_1, \dots, b^Y_{\ell_Y}\} = \varnothing.$$ 
But then the boolean combination  $\sigma(\bar{a}^X \cup \bar{a}^Y ;  \bar{b}^X \cup \bar{b}^Y) \subseteq^* X \cap Y =^* \varnothing$, contradicting the independence of $\I$.
\end{proof}

It is well-known that Cohen forcing adds perfect almost disjoint and perfect almost covering trees of Cohen reals
(see~\cite[Lemma 10.29]{MillerPI11} or~\cite[proof of Lemma 3.6.23]{BaJu95}). We include a proof for the sake of completeness.

 \begin{lemma} \label{e} Let $[s]$ be a basic open set and $c \in [s] $ a Cohen real over the ground model $V$. Then in $V[c]$ there exists a perfect almost disjoint set and a perfect almost covering set of Cohen reals over $V$, contained inside $[s]$. \end{lemma}
 
\begin{proof} For simplicity assume that $[s] = \dw$, and we only prove the case with almost disjoint trees since the other case is similar. Let $\IP$ denote the partial order consisting of finite trees $T \subseteq \dlw$ with the  property that  there are  $0= k_0 < k_1 < \dots < k_\ell$ such that $T \subseteq 2^{\leq k_\ell}$ and for every $i < \ell$, there is \emph{at most one} $t \in T$ where $t \till [k_i, k_{i+1})$ is not constantly 0. The trees are ordered by end-extension.
	
	\p If $G$ is $\IP$-generic, let $T_G$ denote the naturally defined limit of the trees in $G$. By a standard genericity argument, $[T_G]$ must be perfect.  Given any two branches $x, y \in [T_G]$, the construction ensures that for all $n$ after the point where $x$ and $y$ split, either $x(n) = 0$ or $y(n) = 0$, therefore  $x$ and $y$ are almost disjoint. To show that every $x \in [T_G]$ is Cohen over $V$, let $D$ be Cohen-dense and $T \in \IP$ fixed. Enumerate all terminal nodes of $T \subseteq 2^{\leq k_\ell}$ by $\{t_1, \dots, t_j\}$. Extend $T$ to $T' \subseteq 2^{\leq k_{\ell + j}}$ such that each terminal node $t_i \in 2^{k_\ell}$ gets extended to $t_i' \in 2^{k_{\ell + j}}$ such that $t_i '$ is constantly zero on all intervals except $[k_{\ell + i -1 }, k_{\ell + i})$ and $t_i ' \till k_{\ell + i} \in D$. Thus every branch of $T_G$ must meet $D$.
	
\p Since $\IP$ is countable, it is isomorphic to Cohen forcing. Therefore, if $V[c]$ is a Cohen extension of $V$, it is also a $\IP$-generic extension of $V$, so there exists a perfect almost disjoint set $[T_G]$ of Cohen reals. \end{proof}

\begin{corollary}  \label{cora}
Let $[s]$ be basic open and assume $A$ is comeager in $[s]$. Then $A \cap [s]$ contains a perfect a.d set and a perfect a.c. set.
\end{corollary}

\begin{proof} Let $M$ be a countable model containing $A$. Apply the previous lemma in $M[c]$ and note that the perfect a.d. and perfect a.c. sets of Cohen reals must
be contained in $A \cap [s]$ by comeagerness.
\end{proof}

\begin{corollary}  \label{corb}
Let $\Gamma$ be a pointclass closed under existential quantification over reals. Then $\Gamma (\IC)$ implies that there is no m.i.f. in $\Gamma$.
In particular,  $\SIGMA^1_n(\IC) \Rightarrow \nexists \SIGMA^1_n\text{-m.i.f.}$
\end{corollary}

\begin{proof}
Immediate using Lemma~\ref{b} and Corollary~\ref{cora}.
\end{proof}

Using Lemmata~\ref{b} and~\ref{e}, we can complete the proof of the theorem.

\begin{proof}[Proof of Theorem \ref{d}]
Let $W := V^{\IC_{\kappa}}$ (for any $\kappa > \omega$), and let $A$ be a  set in $W$ defined by a formula $\Phi(x)$ with real or ordinal parameters, w.l.o.g. all of which are in $V$.  In $W$, let  $c$ be Cohen over $V$, and assume w.l.o.g. that $\Phi(c)$. Then $V[c] \models $ ``$p \Vdash_{{\IQ}} \Phi(\check{c})$'', where ${\IQ}$ is the remainder forcing leading from $V[c]$ to $W$ and $p$ is some ${\IQ}$-condition. However, since $\IC_{\kappa}$ is the product forcing, ${\IQ}$ is isomorphic to $\IC_{\kappa}$. Moreover, since $\IC_{\kappa}$ is homogeneous we can assume that $p$ is the trivial condition, hence we really have:
$$V[c] \models \text{``}\Vdash_{{\IC}_{\kappa	}} \Phi(\check{c})\text{''}$$
Let $[s]$ be a Cohen condition with $c \in [s]$  forcing this statement in $V$. By Lemma \ref{e}, first we find a perfect a.d.\  tree $T$ with  $T \in V[c]$ , $[T] \subseteq [s]$  and such that all $x \in [T]$ are Cohen over $V$. Note that this fact remains true in $W$, since ``being a perfect set of Cohen reals'' is upwards absolute. Now, for any such $x \in [T]$ (in $W$), we have that $x \in [s]$, and therefore $V[x]$ satisfies whatever $[s]$ forces, in particular
$$V[x] \models \text{``}\Vdash_{{\IC}_{\kappa}} \Phi(\check{x})\text{''}$$
But, again,  the remainder forcing leading from $V[x]$ to $W$ is isomorphic to $\IC_{\kappa}$, and it follows that $W \models  \Phi(x)$.

\p  Similarly, we also find a perfect a.c.\ tree $S$ with exactly the same properties. Thus $A$ satisfies both $\Sadac$ and $\Sacad$, and the rest follows by Lemma \ref{b}. \end{proof}

\newcommand{\id}{{\rm id}}
\newcommand{\FFI}{{{\rm FF}(\I)}}


\section{$\PI^1_1$ m.i.f. in the Sacks model}

In contrast to the above, this section is devoted to the following result:

\begin{theorem} \label{mainsackstheorem} In the countable-support iteration of Sacks forcing, as well as the countable-support product of Sacks forcing, starting from $L$, there exists a $\PI^1_1$ m.i.f. \end{theorem}

As a consequence, we obtain $\Con(\exists  \PI^1_1$-m.i.f.\ of size $< \cont)$, and in fact even $\Con(\exists \PI^1_1$-m.i.f.\ $+$ $\mathfrak{i} < \cont)$. 
Another consequence is the consistency of $\exists \PI^1_1$-m.i.f.\ together with ``all $\SIGMA^1_2$ sets have the Marczewski-property'' (where $X \subseteq \dw$ has the \emph{Marczewski-property} if every perfect set $P$ contains a perfect subset $P'$ with $P' \subseteq X$ or $P' \cap X = \varnothing$), see \cite[Theorem 7.1]{BrLo99}.

The construction we use appeared implicitly in  \cite{ShelahCon} where, among other things, a forcing notion $\IP$ for generically adding a Sacks-indestructible m.i.f. was isolated. These ideas were elaborated and studied further in \cite{FischerMontoya}. Here we show that the combinatorics of this forcing can also be used to explicitly define a Sacks-indestructible m.i.f. in a model of CH, and that in $L$, such a Sacks-indestructible m.i.f. can be defined in a $\SIGMA^1_2$-fashion. We start by recalling some technical definitions from \cite{FischerMontoya}.

To reduce cumbersome notation, in this section the following will be useful:

\begin{notation} If $\I \subseteq \wuw$ then \begin{itemize}
\item $\FFI := \{h: \I \to 2 \; \mid \; |\dom(h)|<\omega\}$, and
\item For $h \in \FFI$ we write
$$\sigma(h) := \bigcap \{A \mid A \in \dom(h) \; \land \; h(A) = 1\} \cap \bigcap \{\omega \setminus A \; \mid \;   A \in \dom(h) \; \land \; h(A) = 0\}.$$
\end{itemize}
\end{notation} 

\newcommand{\E}{{\mathcal{E}}}
\newcommand{\cf}{{\rm cf}}

\begin{definition} An independent family $\I$ is called a \emph{densely maximal independent family} if for all $X \subseteq \omega$, for all $h \in \FFI$ there exists $h' \in \FFI$ with $h' \supseteq h$ such that $\sigma(h') \subseteq^* X$ or $\sigma(h') \cap X =^* \varnothing$. \end{definition}

\begin{definition} Let $\I$ be an independent family. The \emph{density ideal of $\I$} is 
$$\id(\I) := \{X \subseteq \omega \; \mid \;  \forall h \in \FFI \: \exists h' \in \FFI \; (h' \supseteq h \; \land \; \sigma(h') \cap X =^* \varnothing)\}.$$
The dual filter is denoted by  $\id^*(\I)$. \end{definition}

\begin{lemma} \label{ideallemma} If $\I \subseteq \I'$ then $\id(\I) \subseteq \id(\I')$, and if $\I = \bigcup_{\alpha<\kappa}\I_\alpha$ for a regular uncountable $\kappa$, 
where the $\I_\alpha$ form a continuous increasing chain with $|\I_\alpha| < \kappa$, then $\id(\I) = \bigcup_{\alpha<\kappa}\id(\I_\alpha)$ . \end{lemma}

\begin{proof} The first statement is straightforward, and for the second statement, if $X \in \id(\I)$ then we can let  $\alpha<\kappa$ be the least ordinal closed under the $h \mapsto h'$ operation given by the definition of $\id(\I)$. \end{proof}

Recall that a filter $\mathcal{F}$ on $\omega$ is  a \emph{p-filter} iff for every $\{X_n \mid n<\omega\} \subseteq \F$ there exists  $X \in \F$ with $X \subseteq^* X_n$ for all $n$ (``$X$ is a pseudointersection of the $X_n$'s''). A filter $\mathcal{F}$ on $\omega$ is a \emph{q-filter} if for every partition of $\omega$ into finite sets $\E = \{E_n \mid n<\omega\}$, there is $X \in  \mathcal{F}$ such that $|X \cap E_n| \leq 1$ for all $n$ (``$X$ is a semiselector for $\E$''). A filter $\mathcal{F}$ is a \emph{Ramsey filter} if it is both a p-filter and a q-filter (cf. \cite[Section 4.5.A]{BaJu95}). The main ingredient in our proof is the following result:

\begin{theorem}[{\cite{ShelahCon}, \cite[Corollary 37]{FischerMontoya}}] \label{ramseydense} Let $\I$ be a densely maximal independent family, such that the dual filter $\id^*(\I)$ is generated by a Ramsey filter and the filter of cofinite sets (Fr\'echet filter). Then $\I$ remains maximal after a countable-support iteration of Sacks forcing, as well as a countable-support product of Sacks forcing. \end{theorem}

\begin{definition} Let $\IP$ be the forcing poset of all pairs $(\A, A)$ where $\A$ is a countable independent family, $A \in \wuw$, and for all $h \in \FF(\calA)$, $\sigma(h) \cap A$ is infinite. The ordering is given by $(\A',A') \leq (\A,A)$ iff $\A' \supseteq \A$ and $A' \subseteq^* A$. \end{definition}

In \cite{FischerMontoya, ShelahCon} this forcing was used to generically add a Sacks-indestructible m.i.f.  Here, rather than forcing with $\IP$ we will be using it in a purely combinatorial fashion to construct a Sacks-indestructible m.i.f. in a model of CH, and, in particular, a $\SIGMA^1_2$ m.i.f. in $L$.



The following properties of $\IP$  were proved in \cite{ShelahCon, FischerMontoya}: 

\begin{lemma} \label{method} $\;$\begin{enumerate} \renewcommand{\theenumi}{\alph{enumi}}

\item $\IP$ is $\sigma$-closed.
\item   If $(\calA,A) \in \IP$ then there exists $B \subseteq A$ such that $B \notin \A$ and $(\calA \cup \{B\}, A) \leq (\A,A)$.
\item If $Y \subseteq \omega$ is an arbitrary set, then for every $(\A,A) \in \IP$ there exists $(\B,B) \leq (\A,A)$ such that 
$$\;\;\;\;\;\;\;\;\;\;\forall h \in \FF(\B) \; \exists h' \in \FF(\B) \text{ s.t. } h' \supseteq h \text{ and } \sigma(h') \subseteq^* Y \text { or } \sigma(h') \cap Y =^* \varnothing.$$
\item Let $\E := \{E_n \mid n < \omega\}$ be a partition of $\omega$ into finite sets. Then for every $(\calA,A) \in \bbP$  there is $B\subseteq A$ such that $(\calA, B)\leq (\calA, A)$ and $|B \cap E_n| \leq 1$ for all $n$ (``$B$ is a semiselector for $\calE$.'').
\item For all $(\A,A) \in \IP$, if $X \in \id(\calA)$ then  there is $B$ such that  $(\calA, B) \leq (\calA,A)$  and $B \cap X = \varnothing$.

\end{enumerate}\end{lemma}

\begin{proof} See Proposition 15,  Lemma 17,  Corollary 19 and Lemma 14 from \cite{FischerMontoya}, respectively. \end{proof}

\begin{definition} \label{dd} We call $\{(\calA_\alpha,A_\alpha) \mid \alpha<\omega_1\}$ an \emph{indestructibility tower}, if it is a strictly decreasing sequence of $\IP$-conditions and, letting $\calA :=\bigcup_{\alpha \in\omega_1} \calA_\alpha$, the following four requirements are satisfied: 
\begin{enumerate}
\item For every $Y \subseteq \omega$, for every $h \in \FF(\A)$ there is $h' \in \FF(\A)$ with  $h' \supseteq h$ such that $\sigma(h') \subseteq^* Y \text{ or } \sigma(h') \cap Y =^* \varnothing.$
\item For every partition $\calE := \{E_n \mid n<\omega\}$ of $\omega$ into {finite} sets, there is $\alpha < \omega_1$ such that $|A_\alpha \cap E_n| \leq 1$ for all $n$ ($A_\alpha$ is a  semiselector for $\calE$).
\item For each $\alpha < \omega_1$ there is an infinite $A\subseteq^* A_\al$ such that $A\in\calA_{\al+1} \setminus \A_{\alpha}$.
\item For every  $X\in\id(\calA)$   there is an $\alpha <\omega_1$ such that $X \cap   A_\al =^* \varnothing$. 
\end{enumerate}
\end{definition}

\begin{lemma} If  $\{(\calA_\alpha,A_\alpha) \mid \alpha<\omega_1\}$ is an {indestructibility tower}, then $\calA := \bigcup_{\alpha<\omega_1}\calA_\alpha$ is a m.i.f. which remains maximal after a countable-support iteration and a countable-support product of Sacks forcing. \end{lemma}

\begin{proof} In light of Theorem  \ref{ramseydense} it suffices to show that $\calA$ is a densely maximal family and that $\id^*(\I)$ is generated by a Ramsey filter and the filter of cofinite sets. Dense maximality follows immediately from condition (1) of Definition \ref{dd}. For the second property, we show the following: 

\p \emph{Claim.} {$\id(\calA)$ is generated by $\{\omega \backslash A_\alpha \mid \alpha<\omega_1\}$ and $[\omega]^{{<}\omega}$.}

\begin{proof}[Proof of claim] Since by condition (4) of Definition \ref{dd}, for every $X \in \id(\calA)$ there exists $\alpha$ such that $X \subseteq^* \omega \setminus A_\alpha$, it suffices to show that  $\omega \setminus A_\alpha \in \id(\calA)$ for every $\alpha$. Let $h \in \FF(\A)$ be arbitrary. Let $\beta \geq \alpha$ be such that $h \in \FF(\A_\beta)$. By (3) there is an infinite $B \subseteq^* A_\beta$ such that $B \in \calA_{\beta+1}\setminus \A_\beta$. In particular, $B \notin \dom(h)$, so we can extend $h$ to form  $h' := h \cup \{(B,1)\}$. Then $h' \in \FF(\A_{\beta+1})$, and moreover $\sigma(h') \subseteq B \subseteq^* A_\beta \subseteq^* A_\alpha$. This shows that $\omega \setminus A_\alpha \in \id(\A)$ and completes the proof.  \qedhere (Claim) \end{proof} 

\p Notice that since $\{A_\alpha \mid \alpha<\omega_1\}$ is a tower, the filter it generates is a p-filter. Moreover, by condition (2), it is a q-filter, and thus a Ramsey filter, as we had to show. \end{proof}

\begin{theorem} $\;$ \

\begin{enumerate} 
\item  If CH holds then there exists an indestructibility tower. 
\item If $V=L$ then there exists a $\SIGMA^1_2$-definable indestructibility tower.
\end{enumerate} \end{theorem}

\begin{proof} We give a detailed proof of the first assertion and then show how to adapt it to get a $\SIGMA^1_2$ construction in $L$.

\bigskip \noindent (1)  Let $\{X_\alpha \mid \alpha<\omega_1\}$ enumerate all subsets of $\omega$ and let $\{\E_\alpha \mid \alpha<\omega_1\}$ enumerate all partitions of $\omega$ into finite sets.

\p Let $(\A_0, A_0) \in \IP$ be any condition. At stage $\alpha$, suppose $(\A_\beta, A_\beta)$ for all $\beta \leq \alpha$ has been constructed. The new condition is designed in four steps:

\begin{itemize}
\item Consider the sets $\{X_\beta \mid \beta \leq  \alpha\}$. By repeatedly applying Lemma \ref{method} (3) in countably many steps, followed by $\sigma$-closure which holds due to Lemma \ref{method} (1), we find an extension $(\A'_\alpha, A'_\alpha) \leq (\A_\alpha,A_\alpha)$ such that, for all $\beta \leq \alpha$, for all $h \in \FF(\A_\alpha)$ (not necessarily for all $h \in \FF(\A'_\alpha)$) there exists $h' \in \FF(\A'_\alpha)$, such that $h' \supseteq h$ and $\sigma(h') \subseteq^* X_\beta$ or $\sigma(h') \cap X_\beta =^* \varnothing$.

\item Consider the partition $\E_\alpha = \{E_\alpha^n \mid n<\omega\}$. 
By Lemma \ref{method} (4) we find an extension  $(\A''_\alpha, A''_\alpha) \leq (\A'_\alpha, A'_\alpha)$ such that $|A''_\alpha \cap E^n_\alpha| \leq 1$ for all $n$ ($A''_\alpha$ is a semi-selector for $\E_\alpha$).

\item Consider (again) the sets $\{X_\beta \mid \beta \leq \alpha\}$. By repeatedly applying Lemma \ref{method} (5) in countably many steps, followed by $\sigma$-closure, we find a further extension $(\A'''_\alpha, A'''_\alpha) \leq (\A''_\alpha,A_\alpha'')$ such that, for every $\beta$, if $X_\beta \in \id(\A_\alpha)$ (which  implies that $X_\beta \in \id(\tilde{A})$ for any $\tilde{A}$ extending $\A_\alpha$), then $A_\alpha''' \cap X_\beta =^* \varnothing$.

\item Finally, use Lemma \ref{method} (2) to find a $B \subseteq^* A'''_\alpha$, such that $B \notin \A'''_\alpha$, and $(\A'''_\alpha \cup \{B\}, A'''_\alpha)$ is a condition. We let $(\A_{\alpha+1}, A_{\alpha+1})$ be that condition. 
\end{itemize} 
This completes the construction of the induction step (in steps 2 and 3 we could in fact have taken $\A'''_\alpha = \A''_\alpha = \A'_\alpha$ but that is not relevant). At limit stages $\lambda$, use $\sigma$-closure to again find a condition $(\A_\lambda, A_\lambda)$ which extends all $(\A_\alpha,A_\alpha)$ for $\alpha<\lambda$. 

\p It is now easy to verify that $\{(\A_\alpha,A_\alpha) \mid \alpha<\omega_1\}$ satisfies conditions (1)--(4), where for (4) we use the fact that if $X \in \id(\A)$ then $X \in \id(\A_\alpha)$ for some $\alpha<\omega_1$, see Lemma \ref{ideallemma}.

 \bigskip \noindent (2) If   $V=L$ then repeat the same proof, but additionally, pick the canonical well-order $<_L$ of the reals of $L$ to well-order the sequences  $\{X_\alpha \mid \alpha<\omega_1\}$ and  $\{\E_\alpha \mid \alpha<\omega_1\}$. At each step $\alpha$ of  the construction, the preceding proof shows how to find  an $(\A_{\alpha+1}, A_{\alpha+1})$ satisfying certain requirements. Now, we make sure to always pick the  \emph{$<_L$-least} condition  $(\A_{\alpha+1}, A_{\alpha+1})$  satisfying the same requirements. 

\p This way, it follows that the construction at each step $\alpha$ only depends on the  preceding $\beta \leq \alpha$ and is thus absolute between $L$ and an  $L_\delta$ for some appropriate $\delta < \omega_1$. More precisely, if $\mathfrak{A} = \{(\A_\alpha,A_\alpha) \mid \alpha<\omega_1\}$,  then there is a formula $\Phi$ defining $\mathfrak{A}$ in an absolute way, i.e., $(\A,A) \in \mathfrak{A}$ iff $\Phi(\A,A)$ iff there exists $\delta < \omega_1$ such that $L_\delta \models \Phi(\A,A)$. 

\p Let $\ZFC^*$ be a sufficiently large fragment of $\ZFC$ such that if a transitive model $M$ satisfies $\ZFC^* + V=L$ then $M = L_\xi$ for some $\xi$. Now we can write  $\Phi(\A,A)$ iff $\exists E \subseteq \omega \times \omega$ such that \begin{itemize}
\item $E$ is well-founded,
\item $(\omega,E) \models \ZFC^* + V=L$,

\item $(\omega,E) \models \Phi(\pi^{-1}(\A,A))$, where $\pi: (\omega,E) \cong (M,\epsilon)$ is the transitive collapse of $(\omega,E)$. \end{itemize}
By standard methods (cf. \cite[Proposition 13.8 ff.]{Kanamori}) the two latter statements are arithmetic and well-foundedness  is $\PI^1_1$. Thus $\Phi(\A,A)$ is equivalent to a $\SIGMA^1_2$ statement. \end{proof}

\begin{proof}[Proof of Theorem \ref{mainsackstheorem}]  Let $\mathfrak{A} = \{(\A_\alpha,A_\alpha) \mid \alpha<\omega_1\}$ be a $\SIGMA^1_2$-definable indestructibility tower in $L$. If $V$ is the extension in the iteration/product of Sacks forcing, then $\A = \bigcup_{\alpha<\omega_1}  \A_\alpha$ is still a maximal independent family with a $\SIGMA^1_2$ definition. By Theorem \ref{equivalence}, there exists a $\PI^1_1$ m.i.f. as well. \end{proof}

\begin{remark} In Shelah's proof of the consistency of $\mathfrak{i}<\mathfrak{u}$ \cite{ShelahCon} a forcing closely related to Sacks was used, which increases $\mathfrak{u}$ as well as the continuum (note that in the Sacks model $\mathfrak{u} < \cont$). By a slight modification of the method in this section, it is easy to construct a  $\SIGMA^1_2$ m.i.f. which is not only Sacks-indestructible, but indestructible by the poset from \cite{ShelahCon}. This shows that the witness for the m.i.f.  in the proof of the consistency of $\mathfrak{i} < \mathfrak{u}$ can in fact be $\PI^1_1$-definable. 
\end{remark}


\section{$\aleph_1$-Borel and $\aleph_1$-closed m.i.f's} \label{cardinals}

The question of definable m.i.f's is closely related to questions concerning certain cardinal invariants (compare with \cite{BrendleKhomskiiMad}).

\begin{Def} $\;$  \

\begin{enumerate}

\item $\ii$ is the least size of a m.i.f.
\item $\iicl$ is the least $\kappa$ such that there exists a collection $\{C_\alpha \mid \alpha < \kappa\}$, where each $C_\alpha$ is a \emph{closed} independent family, and $\bigcup_{\alpha<\kappa} C_\alpha$ is a m.i.f.

\item $\iiB$ is the least  $\kappa$ such that there exists a collection $\{B_\alpha \mid \alpha < \kappa\}$, where each $B_\alpha$ is a \emph{Borel} independent family, and $\bigcup_{\alpha<\kappa} B_\alpha$ is a m.i.f.
\end{enumerate} \end{Def}

It is clear that $\iiB \leq \iicl \leq \ii$. It is also known that $\mathfrak{r} \leq \ii$,  $\mathfrak{d} \leq \ii$~\cite[Proposition 8.12 and Theorem 8.13]{HandbookCardinals}, and $\non (\M) \leq \ii$~\cite[Theorem 3.6]{BHH04},
where $\mathfrak{d}$,  $\mathfrak{r}$, and $\non (\M)$ denote the dominating and reaping numbers, and the smallest size of a nonmeager set, respectively. 
Notice that if $\iiB > \aleph_1$, then there are no $\SIGMA^1_2$ m.i.f.'s (since $\SIGMA^1_2$-sets are $\aleph_1$-unions of Borel sets). 
$\cov (\M)$ is the least cardinality of a family of meager sets covering the real line. 

\begin{theorem} $\cov(\M) \leq \iiB$. \end{theorem}

\begin{proof} Let $\kappa < \cov(\M)$ and let $\{B_\alpha \mid \alpha<\kappa\}$ be a collection of Borel independent families. We need to show that $\I := \bigcup_{\alpha<\kappa} B_\alpha$ is not maximal. 

\p Suppose otherwise, and for every finite $E \subseteq \kappa$ define
$$ \hskip-0.75cm H_E := \{X \mid \exists \bar{a}, \bar{b}  \in \bigcup_{\alpha \in E} B_\alpha \;  \text{ s.t. }  \sigma(\bar{a}; \bar{b}) \subseteq^* X\}$$
$$ K_E := \{X \mid \exists \bar{a}, \bar{b}  \in  \bigcup_{\alpha \in E} B_\alpha   \;  \text{ s.t. }     \sigma(\bar{a}; \bar{b}) \cap X =^* \varnothing\}. $$ 
Notice that by maximality of $\I = \bigcup_{\alpha<\kappa} B_\alpha$, we have
$$\bigcup \{H_E \cup K_E \mid E \in [\kappa]^{<\omega} \} = \wuw.$$
Since $\kappa < \cov(\M)$, there must exist a finite $E \subseteq \kappa$ such that $H_E \cup K_E \notin \M$. Suppose  $H_E \notin \M$:  since $H_E$ is analytic, there exists a basic open $[s]$ with $[s] \subseteq^* H_E$ (where $\subseteq^*$ means ``modulo meager''). By Corollary \ref{cora}, we can construct a  perfect a.d.\ tree $T$ with $[T]  \subseteq H_E$. But then, by the argument from Lemma \ref{b}, it follows that $\bigcup_{\alpha \in E} B_\alpha$ is not independent, contrary to the assumption. Likewise, if $K_E \notin \M$ then using Corollary \ref{cora}, there exists a perfect a.c.\ tree $S$ with $[S] \subseteq K_E$, and the rest is the same. \end{proof}

We end this paper with the following open questions:

\begin{question} $\;$ \begin{enumerate}
\item Is the existence of a $\PI^1_1$ m.i.f. consistent with $\mathfrak{i} > \aleph_1$? Is it consistent with $\dd > \aleph_1$, $\mathfrak{r} > \aleph_1$, or
$\non (\M) > \aleph_1$?
\item What about a $\PI^1_2$ m.i.f.?
\item Is it consistent that $\iicl < \dd$ or $\iiB < \dd$?
\item Is it consistent that $\iicl < \mathfrak{r}$ or $\iiB < \mathfrak{r}$? 
\item Is it consistent that $\iicl < \non (\M)$ or $\iiB < \non (\M)$?
\item Is it consistent that $\iicl < \ii$ or $\iiB < \ii$?
\end{enumerate} \end{question}

We note that a positive answer to either of (3), (4), or (5) implies a positive answer to (6).
Also, (1) is closely related to (3) -- (6): if, e.g., $\iiB \geq \dd$ in $\ZFC$, then the existence of a $\PI^1_1$ m.i.f. implies $\dd = \aleph_1$ because
$\PI^1_1$ sets are $\aleph_1$ Borel. If, on the other hand, $\iiB < \dd$ is consistent, then, by Zapletal's work~\cite{Za04}, this consistency should hold
in the Miller model; that is, a countable support iteration of Miller forcing should preserve a witness for $\iiB = \aleph_1$, and doing
such an iteration over $L$ then would preserve such a witness with a $\SIGMA^1_2$ definition. By Theorem~\ref{equivalence}, the consistency
of the existence of a $\PI^1_1$ m.i.f. with $\dd > \aleph_1$ would hold in the Miller model.

Similar results have been proved for mad families in~\cite{BrendleKhomskiiMad}: while $\bbb \leq \aaa$ in $\ZFC$, $\aaB = \aacl < \bbb$ is consistent
and so is the existence of a $\PI^1_1$ mad family with $\bbb > \aleph_1$. However, definable m.i.f.'s are more easily destroyed than
definable mad families: for example, the fact that ``$\ZF\; + $ there are no mad families'' is equiconsistent with $\ZFC$ has only been proved recently by Horowitz and Shelah~\cite{HSta}, and it is a well-known open question whether $\AD$ implies that there are no mad families (whereas for maximal independent families, both statements are an easy consequence of Corollary \ref{corb}).


\bibliographystyle{amsplain}
\bibliography{../../../10_Bibliography/Khomskii_Master_Bibliography}{}

\end{document}

%% file: newcommands.tex
\newcommand{\p}{\medskip \noindent}

\newcommand{\PP}{\mathscr{P}}
\newcommand{\cont}{{2^{\aleph_0}}}

\newcommand{\till}{{\upharpoonright}}

\newcommand{\ZF}{{\rm \sf ZF}}
\newcommand{\ZFC}{{\rm \sf ZFC}}

\newcommand{\AD}{{\rm \sf AD}}

\newcommand{\Con}{{{\rm Con}}}

\newcommand{\SIGMA}{\boldsymbol{\Sigma}}
\newcommand{\DELTA}{\boldsymbol{\Delta}}
\newcommand{\PI}{\boldsymbol{\Pi}}

\newcommand{\dw}{2^\omega}
\newcommand{\dlw}{2^{<\omega}}

\newcommand{\wuw}{[\omega]^\omega}

\newcommand{\IP}{\mathbb{P}}
\newcommand{\IQ}{\mathbb{Q}}
\newcommand{\B}{\mathcal{B}}
\newcommand{\F}{\mathcal{F}}

\newcommand{\dom}{{\rm dom}}

\newcommand{\M}{\mathcal{M}}

\newcommand{\IR}{\mathbb{R}}

\newcommand{\IC}{\mathbb{C}}

\newcommand{\cov}{{\rm cov}}

\newcommand{\non}{{\rm non}}

\newcommand{\bb}{\mathfrak{b}}
\newcommand{\dd}{\mathfrak{d}}

\newtheorem{Thm}{Theorem}[section]

\theoremstyle{definition}
\newtheorem{Def}[Thm]{Definition}

\newcommand{\A}{\mathcal{A}}
\newcommand{\I}{\mathcal{I}}